\documentclass{amsart}

\usepackage{amscd, amssymb, amsmath, amsthm}
\usepackage{mathabx} 
\usepackage{enumerate, latexsym, mathrsfs}
\usepackage{xypic}
\usepackage[
		bookmarks=true, bookmarksopen=true,%
    bookmarksdepth=3,bookmarksopenlevel=2,%
    colorlinks=true,%
    linkcolor=blue,%
    citecolor=blue]{hyperref}
\newtheorem{theorem}{Theorem}[section]
\newtheorem{lemma}[theorem]{Lemma}

\newtheorem{corollary}[theorem]{Corollary}
\newtheorem{conjecture}[theorem]{Conjecture}

\theoremstyle{definition}

\newtheorem{notation}[theorem]{Notation}

\newtheorem{remark}[theorem]{Remark}

\numberwithin{equation}{section}

\newcommand{\Natural}{{\mathbb N}}
\newcommand{\Real}{{\mathbb R}}
\newcommand{\Rational}{{\mathbb Q}}

\newcommand{\Integral}{{\mathbb Z}}

\newcommand{\ud}{{\mathrm{d}}}

\title[{Euler class one fillable}]{
On the Euler class one conjecture\\ for fillable contact structures}
     
\author[Yi Liu]{%
        Yi Liu} 
\address{%
        Beijing International Center for Mathematical Research, Peking University\\
				Beijing 100871, China P.R.} 
\email{%
    liuyi@bicmr.pku.edu.cn}

\thanks{Partially supported by NSFC Grant 11925101, 
and National Key R\&D Program of China 2020YFA0712800}

\date{%
 \today} 

\begin{document}

\begin{abstract}
	In this paper, it is proved that
	every oriented closed hyperbolic $3$--manifold $N$ 
	admits some finite cover $M$ with the following property.
	There exists some even lattice point $w$ 
	on the boundary of the dual Thurston norm unit ball of $M$,
	such that $w$ is not the real Euler class 
	of any weakly symplectically fillable contact structure on $M$.
	In particular, 
	$w$ is not the real Euler class of any transversely oriented, taut foliation on $M$.
	This supplies 
	new counter-examples to Thurston's Euler class one conjecture.
\end{abstract}

\maketitle

\section{Introduction}
For any orientable compact $3$--manifold $M$, 
the dual Thurston norm unit ball
is a well-defined compact polytope in $H^2(M,\partial M;\Real)$.
It is central symmetric about the origin, and possibly of positive codimension.
The polytope is the convex hull of finitely many vertices in the integral lattice.
When $M$ is closed,
the vertices are all even lattice points 
(that is, elements that are divisible by $2$ in $H^2(M;\Integral)$ modulo torsion).

If $M$ is oriented, and closed, and irreducible,
then the real Euler class of any transversely oriented taut foliation on $M$
must be an even lattice point in the dual Thurston norm unit ball.
Thurston established this result,
and posed the following partial converse.

\begin{conjecture}[Euler Class One]\label{ecoc_taut}
For any oriented, closed, irreducible $3$--manifold $M$,
every even lattice point in $H^2(M;\Real)$ of dual Thurston norm $1$
is the real Euler class of some transversely oriented taut foliation on $M$.
\end{conjecture}

The special case of Conjecture \ref{ecoc_taut}
regarding vertices of the dual Thurston norm unit ball holds true.
In fact, Gabai proves that every vertex can be realized 
with some transversely oriented taut foliation 
(with smooth leaves and continuous holonomy),
using sutured manifold decompositions;
see \cite[Section 3]{Gabai--Yazdi} for detail.

The general case of Conjecture \ref{ecoc_taut} turns out to be false.
In fact, Yazdi constructs explicit counter-examples 
among oriented closed hyperbolic $3$--manifolds 
of first Betti number $2$ \cite{Yazdi_ecoc}.

In this paper, we construct new counter-examples to Conjecture \ref{ecoc_taut}
by using finite covers of an arbitrary oriented closed hyperbolic $3$--manifold.
Our construction shows that Conjecture \ref{ecoc_taut} fails quite often.
Our main result is as follows.

\begin{theorem}\label{main_ecoc_fillable}
For every oriented closed hyperbolic $3$--manifold $M$,
there exists some connected finite cover $\tilde{M}$ of $M$,
and some even lattice point $\tilde{w}\in H^2(\tilde{M};\Real)$ of dual Thurston norm $1$,
such that $\tilde{w}$ is not the real Euler class
of any weakly symplectically fillable contact structure on $\tilde{M}$.
\end{theorem}

\begin{corollary}\label{ecoc_taut_corollary}
Every oriented closed hyperbolic $3$--manifold has a finite cover 
for which Conjecture \ref{ecoc_taut} does not hold.
\end{corollary}

Corollary \ref{ecoc_taut_corollary} is an immediate consequence of Theorem \ref{main_ecoc_fillable}.
In fact, transversely oriented taut foliations 
can always be perturbed (as plane distributions)
into weakly symplectically fillable contact structures,
by well-known theorems 
due to Eliashberg and Thurston, and due to Etnyre (see Section \ref{Subsec-contact}).

Theorem \ref{main_ecoc_fillable} and Corollary \ref{ecoc_taut_corollary}
supply new counter-examples to Conjecture \ref{ecoc_taut}.
Indeed, as shown by Sivek and Yazdi,
the counter-example cohomology classes in Yazdi's counter-example manifolds 
can be realized by 
weakly symplectically fillable contact structures \cite[Theorem 1.1]{Sivek--Yazdi}.

By iterating the construction in Theorem \ref{main_ecoc_fillable},
one can easily produce new non-realizable cohomology classes in higher finite covers.
However,
I do not know if an obtained non-realizable cohomology class
would turn into a realizable one,
after pulling back to some higher finite cover.
Non-realizable cohomology classes
virtually keep emerging, and they might virtually evaporate.

See Section \ref{Sec-discussion} for more technical discussion
related to other unsolved variants of the Euler class one conjecture.

\subsection*{Ingredients}
In the rest of the introduction,
we explain the main idea to prove Theorem \ref{main_ecoc_fillable}.

Key to our construction is a deep connection between
the next-to-top term in Heegaard Floer homology for a pseudo-Anosov mapping torus
and the $1$--periodic trajectories of the pseudo-Anosov suspension flow.
This connection has been established by works of
Cotton-Clay \cite{Cotton-Clay_sfh},
and Kutluhan--Lee--Taubes \cite{KLT_hf_hm_i,KLT_hf_hm_ii,KLT_hf_hm_iii,KLT_hf_hm_iv,KLT_hf_hm_v},
and Lee--Taubes \cite{LT_hp_swf}.

On the Heegaard Floer homology side,
Ozsv\'ath and Szab\'o introduce an obstruction
to the existence of a weak symplectic filling $(W,\omega)$
of a connected, closed, transversely oriented contact $3$--manifold $(Y,\xi)$ \cite[Section 4]{OS_genus}.
The obstruction class lives in the hat-flavor Heegaard Floer homology of 
the orientation-reversal $-Y$ with twisted coefficients associated to $\omega$
(see Remark \ref{fillable_implies_nonvanishing_remark} for further information).
It is supported in the direct summand 
at the canonical $\mathrm{Spin}^{\mathtt{c}}$ structure $\mathfrak{s}_\xi$ of $\xi$.

Based on a nonvanishing criterion due to Ozsv\'ath and Szab\'o \cite[Theorem 4.2]{OS_genus},
we duduce a useful simplified version.
Our simplified version says that weak symplectic fillability of $(Y,\xi)$
implies nonvanishing of $\widehat{\mathrm{HF}}(-Y,\mathfrak{s}_\xi)$,
(even if the contact invariant vanishes therein).
See Lemma \ref{fillable_implies_nonvanishing} and Remark \ref{fillable_implies_nonvanishing_remark}.
When $Y$ is the mapping torus $M_f$ of some pseudo-Anosov automorphism $f\colon S\to S$,
and when $\langle c_1(\mathfrak{s}_\xi),[S]\rangle= -\chi(S)-2$,
the nonvanishing of $\widehat{\mathrm{HF}}(-Y,\mathfrak{s}_\xi)$
can be translated into the existence of $1$--periodic trajectories
in some particular first homology class (depending on $\mathfrak{s}_\xi$).
The translation follows from the aforementioned connection.
See Lemma \ref{no_trajectory_implies_vanishing} for the precise statement.
We use these results to establish a preparatory criterion (Lemma \ref{no_fillable}),
so as to facilitate direct application.

On the pseudo-Anosov dynamics side,
our basic tool is the Fried cone of homological directions \cite[{Expos\'e 14}]{FLP_book}.
For any pseudo-Ansov automorphism $f\colon S\to S$,
every periodic trajectory of the pseudo-Anosov suspension flow in the mapping torus $M_f$
determines a ray in $H_1(M_f;\Real)$,
emanating from the origin and passing through the homology class of the periodic trajectory.
The closure of all such rays in $H_1(M_f;\Real)$
forms what is called the Fried cone of homological directions.

The Fried cone is a polyhedral cone.
To every boundary face of the Fried cone,
we can associate 
a finite collection of infinite-index quasiconvex subgroups of $\pi_1(M_f)$.
These subgroups capture all the periodic trajectories
whose homology direction on that face, in a suitable sense
(see Lemma \ref{facial_clusters} for a precise statement).
Base on the virtual compact specialization of closed hyperbolic $3$--manifold groups
due to Agol \cite{Agol_VHC} and Wise \cite{Wise_book},
we are able to make use of the above subgroups,
and construct some finite cover of $M_f$, 
such that the Fried cone of the finite cover 
has a boundary face on which a generic rational homology direction
is vacuum of periodic trajectories.
Precisely, we show that
the face has a scaling-invariant open dense subset,
such that no periodic trajectory represents any rational homology class therein
(see Lemma \ref{virtual_vacuum_directions} for detail).
Virtual constructions with similar ideas to this part
have appeared in former works of the author \cite{Liu_vhsr,Liu_profinite_almost_rigidity}.

With these ingredients, our proof of Theorem \ref{main_ecoc_fillable}
can be outlined as follows.

Given any oriented closed hyperbolic $3$--manifold $M$,
we first obtain a fibered finite cover $M'$, 
by the virtual fibering theorem \cite{Agol_VHC}.
Next, we construct a finite cover $M''$ of $M'$ using Lemma \ref{virtual_vacuum_directions},
such that $M''$ has some Fried cone face with generically vacuum rational homology directions,
as described above.
Finally, we construct a finite cyclic cover $\tilde{M}$ of $M''$.
Passing from $M''$ to $\tilde{M}$ will enable us to promote
a vacuum rational homology direction in $H^2(M'';\Real)\cong H_1(M'';\Real)$
to a genuine candidate cohomology class
$\tilde{w}$ in $H^2(\tilde{M};\Real)\cong H_1(\tilde{M};\Real)$.
The last step relies on a cyclic cover trick as we explain in Lemma \ref{iterate}.
Using our preparatory criterion (Lemma \ref{no_fillable}),
we verify that $(\tilde{M},\tilde{w})$ is as asserted in Theorem \ref{main_ecoc_fillable}.
See Section \ref{Sec-main_proof} for detail.

\subsection*{Organization}
In Section \ref{Sec-preliminary}, 
we recall background facts about Conjecture \ref{ecoc_taut} and Theorem \ref{main_ecoc_fillable}.
In Section \ref{Sec-vanishing_criterion}, 
we establish our preparatory criterion 
for non-existence of fillable contact structures (Lemma \ref{no_fillable}).
In Section \ref{Sec-rare}, 
we establish our major virtual construction \ref{virtual_vacuum_directions}.
In Section \ref{Sec-main_proof},
after establishing a cyclic cover trick (Lemma \ref{iterate}),
we prove Theorem \ref{main_ecoc_fillable}.
In Section \ref{Sec-discussion},
we discuss the relation of 
other analogous Euler class one conjectures with our construction.

\subsection*{Acknowledgement}
The author thanks Yi Ni for valuable communication.
The author thanks Yaoping Xie for helpful comments.

\section{Preliminary}\label{Sec-preliminary}
In this section, we collect preliminary facts 
about the Thurston norm and contact structures.

\subsection{The dual Thurston norm unit ball}\label{Subsec-dual_ball}
We review the Thurston norm,
focusing on properties of the dual Thurston norm unit ball.
See \cite{Thurston_paper_norm} for the reference.

For any finite-dimensional real vector space $V$,
a \emph{seminorm} on $V$ refers to a function $\|\cdot\|\colon V\to [0,+\infty)$
with the properties $\|rv\|=|r|\cdot\|v\|$ and $\|v+w\|\leq\|v\|+\|w\|$,
for all $r\in\Real$ and for all $v,w\in V$.
A \emph{norm} is a seminorm that is nondegenerate, 
namely, $\|v\|>0$ for all nonzero $v\in V$.
Every seminorm $\|\cdot\|$ on $V$ 
determines a dual \emph{partial norm} on the dual vector space
$V^*=\mathrm{Hom}_\Real(V,\Real)$ by the expression
$\|u^*\|^*=\mathrm{sup}\{|\langle u^*, v\rangle|\colon \|v\|=1\}$,
which takes values in $[0,+\infty]$, 
and the subset with finite values forms the linear subspace $W^*$ of $V^*$,
which is dual to the zero-norm subspace $W$ of $\|\cdot\|$, namely,
$W=\{v\in V\colon \|v\|=0\}$ and $W^*=\{u\in V^*\colon u^*|_{W}=0\}$.
For any seminorm $\|\cdot\|$ on $V$, the dual partial norm $\|\cdot\|^*$ on $V^*$
is a norm restricted to the maximal finite subspace $W^*$.
Therefore, the unit ball of $\|\cdot\|$ in $V$ is 
a $0$--codimensional closed subset,
which is symmetric about the origin and 
is invariant under translations parallel to the zero-norm subspace;
the unit ball of $\|\cdot\|^*$ in $V^*$ is 
a bounded closed subset,
which is symmetric about the origin,
and is contained in the finite-norm subspace with codimension $0$.

For any orientable connected compact $3$-manifold $M$,
the \emph{Thurston norm} for $M$ is a seminorm on the real linear space $H_2(M,\partial M;\Real)$,
which we recall as follows.
First, 
for any oriented connected compact surface $\Sigma$, define the \emph{complexity} of $S$
to be the nonnegative integer $\mathrm{max}\{-\chi(S),0\}$.
Then, the complexity of an oriented compacted surface $S$
refers to the sum of the complexity of the connected components of $S$.
For any integral relative $2$--homology class $\Sigma\in H_2(M,\partial M;\Integral)$,
assign $\|\Sigma\|_{\mathrm{Th}}$ to be the minimum complexity of $S$,
among all properly embedded, oriented, compact subsurfaces $(S,\partial S)\subset (M,\partial M)$
which represents $\Sigma$.
Thurston shows that $\|\cdot\|_{\mathrm{Th}}$ on $H_2(M,\partial M;\Integral)$
(exists and) extends uniquely to be a seminorm,
denoted as
$$\|\cdot\|_{\mathrm{Th}}\colon H_2(M,\partial M;\Real)\to [0,+\infty).$$

By generality, we also obtain a dual partial norm, 
which we refer to $\|\cdot\|_{\mathrm{Th}^*}$ as the \emph{dual Thurston norm} of $M$,
and denote as
$$\|\cdot\|_{\mathrm{Th}^*}\colon H^2(M,\partial M;\Real)\to [0,+\infty].$$
We denote the unit balls of the Thurston norm and the dual Thurston norm as 
$$\mathcal{B}_{\mathrm{Th}}(M)\subset H_2(M,\partial M;\Real),$$ 
and 
$$\mathcal{B}_{\mathrm{Th}^*}(M)\subset H^2(M,\partial M;\Real),$$
respectively.

Thurston shows that the dual Thurston norm unit ball $\mathcal{B}_{\mathrm{Th}^*}(M)$
is a polytope, which is the convex hull in $H^2(M,\partial M;\Real)$
of finitely many points in the integral lattice,
(that is, $H^2(M,\partial;\Integral)$ modulo torsion).
In dual terms,
this is to say that the Thurston norm unit ball $\mathcal{B}_{\mathrm{Th}}(M)$
is the intersection of finitely many half-spaces in $H_2(M,\partial M;\Real)$
defined by inequalities of the form $l_i(\Sigma) \leq 1$, for $i=1,\cdots,k$,
where $l_i$ are linear functionals on $H_2(M,\partial M;\Real)$
over $\Integral$ with respect to the integral lattice.
When $M$ is closed, 
the vertices of $\mathcal{B}_{\mathrm{Th}^*}(M)$ are all even lattice points,
reflecting the fact
the oriented closed surfaces all have even Euler characteristics.

When $M$ fibers over a circle as a surface bundle,
the bundle projection $\pi\colon M\to S^1$ 
represents a primitive integral cohomology class $\phi_\pi\in H^1(M;\Integral)$,
upon fixing an orientation of $S^1$. 
The Poincar\'e dual is a relative homology class 
$\mathrm{PD}(\phi_\pi)\in H_2(M,\partial M;\Integral)$,
upon fixing an orientation of $M$.
In this case,
Thurston shows that $\phi_\pi$ lies in the open cone 
over a unique codimension--$1$ open face of $\partial \mathcal{B}_{\mathrm{Th}}(M)$,
and every primitive integral cohomology class $\psi\in H^1(M;\Integral)$ 
with $\mathrm{PD}(\psi)$ in that cone represents a fibering of $M$ over a circle.
This open face of $\mathcal{B}_{\mathrm{Th}^*}(M)$ is called the \emph{fibered face}
determined by $\phi_\pi$.
The dual object in $\partial \mathcal{B}_{\mathrm{Th}^*}(M)$ is a vertex.
In fact, it is exactly the (relative) real Euler class of 
the oriented plane distribution over $M$,
as specified by the tangent spaces of the surface fibers
(relative to the obvious trivialization on the boundary).

When $M$ is an orientable closed hyperbolic $3$--manifold,
the Thurston norm on $H_2(M;\Real)$ is nondegenerate.
In this case, both $\mathcal{B}_{\mathrm{Th}^*}(M)\subset H^2(M;\Real)$ 
and $\mathcal{B}_{\mathrm{Th}}(M)\subset H_2(M;\Real)$ 
are $0$--codimensional compact polytopes.

\begin{lemma}\label{dual_ball_cover}
Let $M$ be an oriented compact $3$--manifold.
Suppose that $\kappa\colon M'\to M$ is a finite covering projection.
Denote by $\kappa^*\colon H^2(M;\Real)\to H^2(M';\Real)$ the pullback homomorphism,
and by $\kappa_!\colon H^2(M';\Real)\to H^2(M;\Real)$ the umkehr homomorphism 
$u\mapsto \mathrm{PD}(\kappa_*(\mathrm{PD}(u)))$.
\begin{enumerate}
\item
For every closed face $F$ of $\mathcal{B}_{\mathrm{Th}^*}(M)$,
there is a unique closed face $F'$ of $\mathcal{B}_{\mathrm{Th}^*}(M')$,
such that $F'$ is equal to the intersection of
$\mathcal{B}_{\mathrm{Th}^*}(M')$ with the preimage $\kappa_!^{-1}([M':M]\cdot F)$.
Here, $[M':M]\cdot F$ denotes the rescaled face 
of the rescaled polytope $[M':M]\cdot \mathcal{B}_{\mathrm{Th}^*}(M)$.
\item 
The image $\kappa^*(F)$ in $\mathcal{B}_{\mathrm{Th}^*}(M')$ is contained in $F'$,
and $\kappa^*(\partial F)$ is contained in $\partial F'$.
Hence, $F'$ is 
the minimal closed face of $\mathcal{B}_{\mathrm{Th}^*}(M')$ to contain $\kappa^*(F)$.
\end{enumerate}
\end{lemma}

\begin{proof}
The Thurston norm satisfies the relations
$\|\kappa_*(\Sigma')\|_{\mathrm{Th}}\leq \|\Sigma'\|_{\mathrm{Th}}$, for all $\Sigma'\in H_2(M',\partial M';\Real)$,
and 
$\|\kappa^!(\Sigma)\|_{\mathrm{Th}}=[M':M]\cdot \|\Sigma\|_{\mathrm{Th}}$, for all $\Sigma\in H_2(M,\partial M;\Real)$,
due to deep theorems of Gabai \cite[Corollaries 6.13 and 6.18]{Gabai_taut}.
They imply the dual relations
$\|\kappa^*(u)\|_{\mathrm{Th}^*}\leq \|u\|_{\mathrm{Th}^*}$, for all $u\in H^2(M,\partial M;\Real)$,
and 
$\|\kappa_!(u')\|_{\mathrm{Th}^*}=[M':M]\cdot \|u'\|_{\mathrm{Th}}$, for all $u'\in H^2(M',\partial M';\Real)$.
Note also $\kappa_!\circ\kappa^*=[M':M]\cdot\mathrm{id}$ on $H^2(M,\partial M;\Real)$.
The main assertions in Lemma \ref{dual_ball_cover} follow immediately as
the geometric interpretation of the dual relations.

The last implication, regarding minimality of $F'$, is equivalent to
saying that $\kappa^*(F)$ contains some interior point of $F'$.
This is the most obvious if $M'$ is regular over $M$.
In that case, it is well-known that 
$\kappa^*(H^2(M,\partial M;\Real))$
consists exactly of the elements of $H^2(M',\partial M';\Real)$
that are fixed under the induced action of the deck transformation group;
see \cite[Proposition 3G.1]{Hatcher_AT} 
(applying with the cohomological long exact sequence for relative pairs).
The barycenter of the vertices of $F'$
is fixed under the deck transformation group,
so it is an interior point of $F'$ contained in $\kappa^*(F)$, as desired.
The general case can be reduced to the regular case
by passing to a finite cover of $M'$ regular over $M$.
\end{proof}

\subsection{Weak symplectic fillability}\label{Subsec-contact}
We recall preliminary facts in $3$--dimensional contact topology,
focusing on weakly symplectically fillable contact structures.
See the textbook \cite{OzbS_book} of Ozbagci and Stipsicz for a reference.

Let $M$ be an oriented closed smooth $3$--manifold.
A (positive) \emph{transversely oriented contact structure} on $M$
can be characterized as a transversely oriented plane distribution $\xi$,
such that $\xi$ is equal to the kernel of some smooth $1$--form $\alpha$
everywhere on $M$,
and
such that $\alpha\wedge\ud\alpha$ is a volume form agreeing with the orientation of $M$.
Here, $\alpha$ is not included as part of the data.
The orientation of $M$ and the transverse orientation of $\xi$
together induce a unique, compatible orientation of the planes in $\xi$.
The \emph{Euler class} 
$$e(\xi)\in H^2(M;\Integral)$$ 
refers to the Euler class of $\xi$ as an oriented plane bundle over $M$, 
with the compatibly induced orientation.
The \emph{real Euler class} of $\xi$ refers to the real reduction of $e(\xi)$,
but we often denote it with the same notation,
only declaring it in $H^2(M;\Real)$.

If $\xi$ is a transversely oriented contact structure on $M$,
then the same plane distribution with the opposite transverse orientation
is also a transversely oriented contact structure (as witnessed by $-\alpha$),
which we denote as $-\xi$. Hence,
$$e(-\xi)=-e(\xi).$$
Both $\xi$ and $-\xi$ are
``negative'' transversely oriented contact structures
on the orientation reversal $-M$,
which are not to be considered in this paper.

For any closed, transversely oriented, contact $3$--manifold $(M,\xi)$,
a \emph{weak symplectic filling} of $(M,\xi)$
refers to any compact symplectic $4$--manifold $(W,\omega)$ 
with $\partial W=M$ (in the oriented sense),
such that $\omega|_\xi$ is positive everywhere on $M$.
We say that $(M,\xi)$ is \emph{weakly symplectically fillable}
if such a $(W,\omega)$ exists.
In this case, 
$(W,-\omega)$ is a weak symplectic filling of $(M,-\xi)$.

Weakly symplectically fillable contact structures arise
naturally from transversely orientable taut folations.
Recall that a codimension--$1$ foliation $\mathscr{F}$ on $M$
(with smooth leaves and continuous holonomy)
is said to be \emph{taut} if every leaf intersects some smoothly immersed loop 
which is transverse to all the leaves.
If $M$ supports a transversely oriented taut foliation $\mathscr{F}$,
then the transversely oriented plane distribution $T\mathscr{F}$
can be perturbed homotopically into a pair of 
weakly symplectically fillable contact structures,
one on $+M$, and the other on $-M$,
by constructions due to 
Eliashberg--Thurston \cite[Proposition 3.2.2]{ET_book} and Etnyre \cite[Corollary 1.2]{Etnyre_fillings}.
In particular, the Euler classes of these transversely oriented contact structures
are the same as the Euler class of $\mathscr{F}$ (with the compatibly induced orientation),
namely, $e(\mathscr{F})=e(T\mathscr{F})$ in $H^2(\pm M;\Integral)$, respectively.
However, 
there are weakly symplectically fillable contact $3$--manifolds
without transversely oriented taut foliations of the identical real Euler class,
for example, see \cite{Sivek--Yazdi}.

Weakly symplectically fillable contact $3$--manifolds are all tight.
A contact $3$--manifold is said to be \emph{tight}
if there are no smoothly embedded disks 
with the tangent planes equal to the contact plane distribution 
everywhere on the boundary circle.
Otherwise, it is \emph{overtwisted},
and a witnessing disk is called an \emph{overtwisted disk}.
The tightness of weakly symplectically fillable contact structures
is originally due to Eliashberg and Gromov \cite{Eliashberg--Gromov};
see also \cite[Theorem 12.1.10]{OzbS_book}.
Tight or weakly symplectically fillable contact structures
may become overtwisted when pulled back to finite covers,
contrasting the case with taut foliations.
There are many examples of
oriented connected closed $3$--manifolds
with transversely oriented, tight contact structures
which do not admit weak symplectic fillings,
see \cite[Section 12.1]{OzbS_book}.

The following well-known fact implies the similar property
regarding the real Euler class of a weakly symplectically fillable contact structure,
or of a taut foliation.
The bound is originally due to Eliashberg \cite[Theorem 4.3.8]{OzbS_book},
and the parity property also follows along the lines of the proof.
We include a quick alternative proof using convex surfaces in contact topology,
for the reader's convenience.

\begin{lemma}\label{bound_and_parity}
Let $M$ be an oriented closed $3$--manifold.
Then, the real Euler class $e(\xi)\in H^2(M;\Real)$ 
of any tight contact structure $\xi$ on $M$
is an even lattice point in the dual Thurston norm unit ball $\mathcal{B}_{\mathrm{Th}^*}(M)$.
\end{lemma}

\begin{proof}
Suppose that $(M,\xi)$ is a closed tight contact $3$--manifold.
For any primitive homology class $\Sigma\in H_2(M;\Integral)$, we can represent $\Sigma$
with an embedded, oriented, closed, Thurston-norm minimizing subsurface $S\subset M$.
Moreover, $S$ can be perturbed isotopically into a convex position
with respect to $\xi$ \cite[Proposition 5.1.6]{OzbS_book}.
We refer to \cite[Chapter 5]{OzbS_book} for an introduction to \emph{convex surfaces},
which have been extremely useful for studying tight contact structures on $3$--manifolds.

For our application, it suffices to know what the output is.
Being a convex surface with respect to $\xi$,
$S$ can be decomposed into two compact subsurfaces $S_+$ and $S_-$,
meeting along their common boundary $\Gamma_S=\partial S_+=\partial S_-$.
The common boundary is a disjoint union of simple closed curves (called the \emph{dividing set}),
and every connected component of $S$ contains at least one of these curves.
Moreover, the following formula holds \cite[Proposition 5.1.15]{OzbS_book}:
$$\langle e(\xi),[S]\rangle=\chi(S_+)-\chi(S_-).$$
The tightness of $\xi$ implies that no connected components of $S_\pm$ are disks 
\cite[Theorem 5.1.20]{OzbS_book}. Therefore, we estimate:
$$|\langle e(\xi),[S]\rangle|\leq
|\chi(S_+)|+|\chi(S_-)|=
(-\chi(S_+))+(-\chi(S_-))=-\chi(S).$$
Since $\Sigma=[S]$ and $\|\Sigma\|_{\mathrm{Th}}=\max(-\chi(S),0)=-\chi(S)$,
we conclude that
$\langle e(\xi),\Sigma\rangle$ is an even integer of absolute value at most $\|\Sigma\|_{\mathrm{Th}}$,
and the same conclusion holds for every primitive $\Sigma\in H_2(M;\Integral)$.
This is equivalent to the assertion that $e(\xi)\in H^2(M;\Real)$ 
is an even lattice point in $\mathcal{B}_{\mathrm{Th}^*}(M)$.
\end{proof}

In $3$--dimensional contact topology,
there are several similar notions of convex fillability in the literature.
Stein fillability and holomorphic fillability are equivalent.
They imply exact symplectic fillability.
Exact symplectic fillability implies strong symplectic fillability.
Strong symplectic fillability implies weak symplectic fillability.
None of these implications are equivalences.
A seemingly more relaxed condition, called weak symplectic semi-fillability,
turns out to be equivalent to weak symplectic fillability.
See \cite[Chapter 12]{OzbS_book}; 
see also \cite{Ghiggini_strongly_fillable} and \cite{Bowden_exact_fillable}.

\section{No one-periodic trajectory, no fillable contact structure}\label{Sec-vanishing_criterion}

In this section, 
we obtain a criterion for nonexistence of fillable contact structures
with certain next-to-top real Euler classes on pseudo-Anosov mapping tori,
in term of absence of certain $1$--periodic trajectories (Lemma \ref{no_fillable}).

We fix some notations for pseudo-Anosov mapping tori,
which are frequently referred to in this paper.
For any orientable connected closed surface $S$ of genus $\geq2$,
a \emph{pseudo-Anosov automorphism} $f\colon S\to S$ refers to
an orientation-preserving homeomorphism, such that there exist 
a pair of measured foliations 
$(\mathscr{F}^{\mathtt{u}},\mu^{\mathtt{u}})$ and $(\mathscr{F}^{\mathtt{s}},\mu^{\mathtt{s}})$
and a constant $\lambda>1$,
and they satisfy the properties $f\cdot(\mathscr{F}^{\mathtt{u}},\lambda\mu^{\mathtt{u}})$
and $f\cdot(\mathscr{F}^{\mathtt{u}},\lambda^{-1}\mu^{\mathtt{u}})$.
We require $\mathscr{F}^{\mathtt{u}}$ and $\mathscr{F}^{\mathtt{s}}$ to be transverse to each other,
except at finitely many common singular points. We require $\mathscr{F}^{\mathtt{u}}$ and $\mathscr{F}^{\mathtt{s}}$
to have prong-type singularities at the singular points, with prong number $\geq3$.
Among the isotopy class of $f$, a representative of this standard form is unique up to topological conjugacy.
See \cite{FLP_book} for an exposition of pseudo-Anosov automorphisms on surfaces;
see also \cite[Section 2]{Liu_vhsr} for a summary of facts 
from the perspective of mapping tori.

\begin{notation}\label{pA_notation}
For any pseudo-Anosov automorphism $f\colon S\to S$ on an orientable connected closed surface of genus $\geq2$,
we adopt the following notations.
\begin{enumerate}
\item
The mapping torus $M_f$ of $f$ is constructed as the quotient of $S\times\Real$ by 
the equivalence relation $(f(x),r)\sim(x,r+1)$ for all $(x,r)\in S\times\Real$.
The suspension flow $\theta_t\colon M_f\to M_f$ is 
induced by the action $(x,r)\mapsto (x,r+t)$ on $S\times\Real$ for all $t\in\Real$.
The distinguished fibered class $\phi_f\in H^1(M_f;\Integral)$ 
is (homotopically) represented by 
the fibering projection $M_f\to S^1$ induced by $(x,r)\mapsto r$,
identifying $S^1$ with $\Real/\Integral$.
\item
When $S$ is oriented, $M_f$ is oriented compatibly
with respect to the oriented surface fibers and forward direction of the suspension flow.
The distinguished Euler class $e_f\in H^2(M_f;\Integral)$
refers to the Euler class of the oriented tangent plane distribution of the surface fibers.
The distinguished $\mathrm{Spin}^{\mathtt{c}}$ structure $\mathfrak{s}_\theta$ on $M_f$
is represented by the velocity vector field $\dot{\theta}$ of $\theta_t$.
Hence, there are relations $e_f=c_1(\mathfrak{s}_\theta)$, 
and $\langle e_f,[S]\rangle=\chi(S)$,
and $\phi_f=\mathrm{PD}([S])$.
\end{enumerate}
\end{notation}

\begin{lemma}\label{no_fillable}
Let $S$ be an oriented connected closed surface of genus $\geq3$,
and $f\colon S\to S$ be a pseudo-Anosov automorphism.
Denote by $e_f\in H^2(M_f;\Real)$ 
the distinguished real Euler class (Notation \ref{pA_notation}).

If $a\in H^2(M_f;\Real)$ is an integral lattice point satisfying $\langle a,[S]\rangle=1$,
and if the Poincar\'e dual $\mathrm{PD}(a)\in H_1(M_f;\Real)$
is not represented by any $1$--periodic trajectory of the suspension flow,
then the even lattice point $-e_f-2a$ is not the real Euler class of 
any weakly symplectically fillable contact structure on $M_f$,
nor is $e_f+2a$.
\end{lemma}

The rest of this section is devoted to the proof of Lemma \ref{no_fillable}.

Our proof is based on the nonvanishing criterion of weak symplectic fillability 
due to Ozsv\'ath--Szab\'o \cite[Theorem 4.2]{OS_genus},
and the isomorphisms between
several $3$--dimensional Floer homology theories due to Lee--Taubes \cite{LT_hp_swf},
and due to Kutluhan--Lee--Taubes \cite{KLT_hf_hm_i,KLT_hf_hm_ii,KLT_hf_hm_iii,KLT_hf_hm_iv,KLT_hf_hm_v},
and the study of the symplectic Floer homology for pseudo-Anosov mapping classes 
due to Cotton-Clay \cite{Cotton-Clay_sfh}. 

From these essential ingredients, 
it is more or less straightforward to derive Lemma \ref{no_fillable}.
For this reason, our proof below is largely expository.
However, 
Lemma \ref{fillable_implies_nonvanishing} seems to be 
a new observation of independent interest (see Remark \ref{fillable_implies_nonvanishing_remark}).

We recall that for any oriented connected closed $3$--manifold $Y$,
there are three flavors of the Heegaard Floer homology 
$\mathrm{HF}^+(Y)$, $\mathrm{HF}^-(Y)$, and $\mathrm{HF}^\infty(Y)$,
which are relatively $\Integral/2\Integral$--graded $\Integral[U]$--modules,
where $U$ is a fixed indeterminant;
there is another flavor $\widehat{\mathrm{HF}}(Y)$
which is a relatively $\Integral/2\Integral$--graded $\Integral$--module
(or a $\Integral[U]$--module with annihilating $U$--action) \cite{OS_hf_invariants}.
More generally, for any module $V$ over the abelian group ring 
\begin{equation}\label{notation_Lambda}
\Lambda=\Integral[H^1(Y;\Integral)],
\end{equation}
there are corresponding versions with twisted coefficients,
denoted as $\underline{\mathrm{HF}}^+(Y;V)$, $\underline{\mathrm{HF}}^-(Y;V)$, $\underline{\mathrm{HF}}^\infty(Y;V)$,
and $\widehat{\underline{\mathrm{HF}}}(Y;V)$.
These are relatively $\Integral/2\Integral$--graded $\Lambda[U]$--modules \cite[Section 8.1]{OS_hf_properties}.
Each of the flavors naturally splits into direct summands
$$\underline{\mathrm{HF}}^\circ(Y;V)=\bigoplus_{\mathfrak{s}}\underline{\mathrm{HF}}^\circ(Y,\mathfrak{s};V),$$
indexed by the $\mathrm{Spin}^{\mathtt{c}}$ structures $\mathfrak{s}$ of $Y$.
The splitting respects the relative grading and the module structure,
and has only finitely many nonvanishing direct summands.
With the regular $\Lambda$--module $V=\Lambda$,
we obtain the universally twisted versions, which are simply denoted as
$$\underline{\mathrm{HF}}^\circ(Y,\mathfrak{s})=\underline{\mathrm{HF}}^\circ(Y,\mathfrak{s};\Lambda).$$

Recall that the $\mathrm{Spin}^{\mathtt{c}}$ structures on $Y$
forms an affine $H_1(Y;\Integral)$.
That is to say, after choosing 
a reference $\mathrm{Spin}^{\mathtt{c}}$ structure $\mathfrak{s}_0$,
any other $\mathrm{Spin}^{\mathtt{c}}$ on $Y$ can be 
written as $\mathfrak{s}_0+\Gamma$ for some unique $\Gamma\in H_1(Y;\Integral)$.
There is a canonical involution $\mathfrak{s}\mapsto\bar{\mathfrak{s}}$
on the space of $\mathrm{Spin}^{\mathtt{c}}$ on $Y$.
The first Chern class $c_1(\mathfrak{s})\in H^2(Y;\Integral)$
can be determined by the relation $\bar{\mathfrak{s}}=\mathfrak{s}-\mathrm{PD}(c_1(\mathfrak{s}))$.
Therefore, $c_1(\bar{\mathfrak{s}})=-c_1(\mathfrak{s})$, 
and $c_1(\mathfrak{s}+\Gamma)=c_1(\mathfrak{s})+2\mathrm{PD}(\Gamma)$.

\begin{remark}\label{orientation_reversal}
We treat any $\mathrm{Spin}^{\mathtt{c}}$ structure $\mathfrak{s}$ 
on an oriented $3$--manifold $Y$ as represented by a nowhere vanishing vector field,
so the first Chern class $c_1(\mathfrak{s})$ refers to
the Euler class of the compatibly oriented transverse plane distribution.
The same vector field also represents a $\mathrm{Spin}^{\mathtt{c}}$ structure
on the orientation reversal $-Y$, which we denote with the same notation $\mathfrak{s}$.
We adopt the notation $H_*(\pm Y;\Integral)$
for $H_*(Y;\Integral)$ together with the specified fundamental class $\pm[Y]\in H_3(Y;\Integral)$,
which also specifies the Poincar\'e duality isomorphism.
Therefore, when $Y$ flips orientation,
$c_1(\mathfrak{s})$ flips sign in $H^2(+Y;\Integral)=H^2(-Y;\Integral)$, 
but the sign of $\mathrm{PD}(c_1(\mathfrak{s}))$ 
does not alter in $H_1(+Y;\Integral)=H_1(-Y;\Integral)$.
The involution $\mathfrak{s}\mapsto\bar{\mathfrak{s}}$ 
corresponds to reversing the direction of any representative vector field.
\end{remark}

Our proof below only involves twisted versions of the hat flavor and the plus flavor.
These two flavors are related by the following exact triangle of $\Lambda$--module homomorphisms:
\begin{equation}\label{exact_triangle_pph}
\xymatrix{
\underline{\mathrm{HF}}^+(Y,\mathfrak{s};V) \ar[r]^-{U} & 
\underline{\mathrm{HF}}^+(Y,\mathfrak{s};V) \ar[d] \\ 
& \widehat{\underline{\mathrm{HF}}}(Y,\mathfrak{s};V) \ar[lu]
}
\end{equation}
Here, the multiplication by $U$ 
and the downward arrow respect the $\Integral/2\Integral$--grading,
and the upper-leftward arrow switches the $\Integral/2\Integral$--grading.

The \emph{rank} of a module $M$ over a (commutative) domain $R$
refers to the dimension of $M\otimes_R\mathrm{Frac}(R)$
over the field of fractions $\mathrm{Frac}(R)$.
If $V$ is finitely generated over $\Lambda$,
then $\widehat{\underline{\mathrm{HF}}}(Y,\mathfrak{s};V)$ 
is also finitely generated over the Noetherian domain $\Lambda$,
so $\widehat{\underline{\mathrm{HF}}}(Y,\mathfrak{s};V)$ has finite rank over $\Lambda$.
It may happen that $\underline{\mathrm{HF}}^+(Y,\mathfrak{s};V)$
has infinite rank over $\Lambda$ when $c_1(\mathfrak{s})$ is torsion.
However, every element of $\underline{\mathrm{HF}}^+(Y,\mathfrak{s};V)$
is annihilated by some sufficiently large power of $U$,
implying $\underline{\mathrm{HF}}^+(Y,\mathfrak{s};V)=0$
if and only if $\widehat{\underline{\mathrm{HF}}}(Y,\mathfrak{s};V)=0$,
by the exact triangle (\ref{exact_triangle_pph}).

\begin{lemma}\label{fillable_implies_nonvanishing}
If $(Y,\xi)$ is an oriented closed contact $3$--manifold 
which is weakly symplectically fillable,
then
$$\mathrm{rank}_\Integral\,\widehat{\mathrm{HF}}(-Y,\mathfrak{s}_\xi)>0.$$
Here, $\mathfrak{s}_\xi$ denotes the canonical $\mathrm{Spin}^{\mathtt{c}}$ structure of $\xi$,
which is represented by any nowhere vanishing vector field transverse to $\xi$ 
agreeing with its prescribed transverse orientation.
\end{lemma}

\begin{proof}
We derive Lemma \ref{fillable_implies_nonvanishing}
from \cite[Theorem 4.2]{OS_genus} due to Ozsv\'ath and Szab\'o.
To be precise, Ozsv\'ath and Szab\'o show that 
if $(W,\omega)$ is a weakly symplectic filling of an oriented connected closed contact $3$--manifold $(Y,\xi)$,
then the twisted Floer homology group 
$\widehat{\underline{\mathrm{HF}}}(-Y,\mathfrak{s}_\xi;\Integral[\Real]^\omega)$
contains a $\Integral$--primitive, $\Integral[\Real]$--nontorsion element 
(see Remark \ref{fillable_implies_nonvanishing_remark} for technical elaboration).
Here, the twisting coefficient is taken to be the abelian group ring $\Integral[\Real]$
treated as the specialization of $\Lambda$ (see (\ref{notation_Lambda}))
induced by the abelian group homomorphism
$\omega_*\colon H^1(Y;\Integral)\to \Real\colon x\mapsto \langle [\omega], \mathrm{PD}(x)\rangle$.
The specialization depends only 
on the restriction of the de Rham cohomology class $[\omega]\in H^2(W;\Real)$ 
to $Y=\partial W$.

Denote by $\Omega=\Integral[\mathrm{Im}(\omega)]$ 
the image of the specialization homomorphism $\Lambda\to\Integral[\Real]$ induced by $\omega_*$.
We observe that $\Integral[\Real]$ is flat over the subdomain $\Omega$.

By definition, $\widehat{\underline{\mathrm{HF}}}(-Y,\mathfrak{s}_\xi)$
is the homology of a chain complex of finitely generated free $\Lambda$--modules
$(\widehat{\underline{\mathrm{CF}}}(-Y,\mathfrak{s}_\xi),\widehat{\underline{\partial}})$,
and $\widehat{\underline{\mathrm{HF}}}(-Y,\mathfrak{s}_\xi;V)$
is the homology of the chain complex of $\Lambda$--modules
$(\widehat{\underline{\mathrm{CF}}}(-Y,\mathfrak{s}_\xi)\otimes_\Lambda V,
\widehat{\underline{\partial}}\otimes 1)$.
Flatness implies
$\widehat{\underline{\mathrm{HF}}}(-Y,\mathfrak{s}_\xi;\Omega)\otimes_{\Omega} \Integral[\Real]
\cong \widehat{\underline{\mathrm{HF}}}(-Y,\mathfrak{s}_\xi; \Integral[\Real]^\omega)$
as $\Lambda$--modules.
In particular, 
the finitely generated $\Lambda$--module
$\widehat{\underline{\mathrm{HF}}}(-Y,\mathfrak{s}_\xi;\Omega)$ 
is not $\Omega$--torsion.
Equivalently, we obtain
\begin{equation}\label{omega_nonvanishing}
\mathrm{rank}_{\Omega}\,\widehat{\underline{\mathrm{HF}}}(-Y,\mathfrak{s}_\xi;\Omega)>0.
\end{equation}

The untwisted version $\widehat{\mathrm{HF}}(-Y,\mathfrak{s}_\xi)$
is the homology of the chain complex of finitely generated free $\Integral$--modules
$(\widehat{\mathrm{CF}}(-Y,\mathfrak{s}_\xi),\widehat{\partial})$
recovered from 
$(\widehat{\underline{\mathrm{CF}}}(-Y,\mathfrak{s}_\xi),\widehat{\underline{\partial}})$
by the trivial specialization $\Lambda\to\Integral$,
(sending every group element in $H^1(Y;\Integral)$ to $1$).
We can also recover $\widehat{\mathrm{CF}}(-Y,\mathfrak{s}_\xi)$
from $\widehat{\underline{\mathrm{CF}}}(-Y,\mathfrak{s}_\xi;\Omega)$
by the trivial specialization $\Omega\to\Integral$.
A semicontinuity argument as usual in homological algebra implies
\begin{equation}\label{omega_semicontinuous}
\mathrm{rank}_\Integral\,\widehat{\mathrm{HF}}(-Y,\mathfrak{s}_\xi)
\geq \mathrm{rank}_{\Omega}\,\widehat{\underline{\mathrm{HF}}}(-Y,\mathfrak{s}_\xi;\Omega).
\end{equation}
In fact, over a basis one may represent the boundary operator
$\widehat{\underline{\partial}}$ of 
$\widehat{\underline{\mathrm{CF}}}(-Y,\mathfrak{s}_\xi;\Omega)$
as a square matrix over $\Omega$, and $\widehat{\partial}$ is the specialized matrix over $\Integral$.
The rank of $\widehat{\partial}$ over $\Rational$ cannot decrease under
sufficiently small perturbation of any specialization $\mathrm{Frac}(\Omega)\to\Rational$.
Meanwhile, the rank under a generic (small) perturbation 
is equal to the rank of $\widehat{\underline{\partial}}$ over $\Omega$.
Therefore, 
the rank of $\widehat{\partial}$ over $\Integral$ is a lower bound 
for the rank of $\widehat{\underline{\partial}}$ over $\mathrm{Frac}(\Omega)$,
and hence, 
the nullity of $\widehat{\partial}$ over $\Integral$ is 
an upper bound for the nullity of $\widehat{\underline{\partial}}$ over $\Omega$.
Then the above inequality follows immediately.

The asserted inequality as in Lemma \ref{fillable_implies_nonvanishing}
follows from the inequalites (\ref{omega_nonvanishing}) and (\ref{omega_semicontinuous}).
\end{proof}

\begin{remark}\label{fillable_implies_nonvanishing_remark}\
\begin{enumerate}
\item 
For any module $V$ over $\Lambda$ (see (\ref{notation_Lambda})), 
Ozsv\'ath and Szab\'o construct an element
$\underline{c}(\xi;V)\in \widehat{\underline{\mathrm{HF}}}(-Y,\mathfrak{s}_\xi;V)$,
which depends only on $(Y,\xi)$
up to a unit in $\Lambda$ 
and up to an isomorphism of $\widehat{\underline{\mathrm{HF}}}(-Y,\mathfrak{s}_\xi;V)$. 
The image of $\underline{c}(\xi;V)$ in $\underline{\mathrm{HF}}^+(-Y,\mathfrak{s}_\xi;V)$ 
is denoted as $\underline{c}^+(\xi;V)$ (see (\ref{exact_triangle_pph})).
Twisting with $\Integral[\Real]^\omega$, 
as in the proof of Lemma \ref{fillable_implies_nonvanishing},
the argument of \cite[Theorem 4.2]{OS_genus} actually obtains
some $\Integral[\Real]$--linear homomorphism 
$\underline{\mathrm{HF}}^+(-Y,\mathfrak{s}_\xi;\Integral[\Real]^\omega)\to\Integral[\Real]$
(to be precise, $\underline{F}^+_{W-B_2}$ therein; 
see also \cite[Section 3.1]{OS_smooth_four} for definition and properties).
That argument shows that 
the image of $\underline{c}^+(\xi;\Integral[\Real]^\omega)$ under the homomorphism 
is nonzero in $\Integral[\Real]$, 
and has coefficient $\pm1$ in its lowest order term
(viewed as a generalized Laurent polynomial allowing real powers of the indeterminant).
It follows that both $\underline{c}^+(\xi;\Integral[\Real]^\omega)$ 
and $\underline{c}(\xi;\Integral[\Real]^\omega)$
are $\Integral$--primitive and $\Integral[\Real]$--nontorsion.
This clarifies an omitted detail in the statement of \cite[Theorem 4.2]{OS_genus}.
\item
Lemma \ref{fillable_implies_nonvanishing} 
does not mean $c(\xi)\neq0$ in $\widehat{\mathrm{HF}}(-Y,\mathfrak{s}_\xi)$.
See Ghiggini \cite{Ghiggini_vanishing_os} 
for vanishing examples with modulo $2$ coefficients.
See Hedden and Ni \cite[Theorem 2.2]{Hedden--Ni} 
for a similar nonvanishing theorem
with respect to taut surfaces.
\item
See Hedden and Tovstopyat-Nelip \cite{Hedden--Tovstopyat-Nelip} for 
the naturality of the Ozsv\'ath--Szab\'o contact invariant $c(\xi)$
with modulo $2$ coefficients, 
and for references regarding more general coefficients.
\end{enumerate}
\end{remark}

\begin{lemma}\label{no_trajectory_implies_vanishing}
Let $S$ be an oriented connected closed surface of genus $\geq3$,
and $f\colon S\to S$ be a pseudo-Anosov automorphism.
Denote by $\mathfrak{s}_\theta$ 
the distinguished $\mathrm{Spin}^{\mathtt{c}}$ structure on $M_f$ (Notation \ref{pA_notation}).

If $\Gamma\in H_1(M_f;\Integral)$ satisfies $\langle \mathrm{PD}(\Gamma),[S]\rangle=1$,
and if $\Gamma$ is not represented by any $1$--periodic trajectory of the suspension flow,
then
$$\widehat{\mathrm{HF}}(M_f,\mathfrak{s}_\theta+\Gamma)=0.$$
\end{lemma}

\begin{proof}
The conclusion $\widehat{\mathrm{HF}}(M_f,\mathfrak{s}_\theta+\Gamma)=0$
is equivalent to $\mathrm{HF}^+(M_f,\mathfrak{s}_\theta+\Gamma)=0$,
as we have explained.

When $S$ has genus $\geq3$, 
(the first Chern class of) 
the $\mathrm{Spin}^{\mathtt{c}}$ structure $\mathfrak{s}_\theta+\Gamma$ is not torsion,
indeed, $\langle c_1(\mathfrak{s}_\theta+\Gamma),[S]\rangle = \chi(S)+2<0$.
In this case,
the plus-flavor Heegaard Floer homology at $\mathfrak{s}_\theta+\Gamma$
is isomorphic to the periodic Floer homology at $\Gamma$,
\begin{equation}\label{hf_to_hp}
\mathrm{HF}^+(M_f,\mathfrak{s}_\theta+\Gamma)\cong\mathrm{HP}(f,\Gamma),
\end{equation}
(as relatively $\Integral/2\Integral$--graded $\Integral$--modules, for simplicity).
To be more precise, 
$\mathrm{HP}(f;\Gamma)$ is actually defined up to natural isomorphism
for any mapping class $[f]\in \mathrm{Mod}(S)$ 
and a homology class $\Gamma\in H_1(M_f;\Integral)$
with $\langle \mathrm{PD}(\Gamma),[S]\rangle>0$.
The case with $\langle \mathrm{PD}(\Gamma),[S]\rangle=1$
is formerly introduced as the symplectic Floer homology.
The isomorphism (\ref{hf_to_hp}) follows from the aforementioned works
due to Kutluhan--Lee--Taubes and due to Lee--Taubes;
we refer the reader to \cite[Appendix A]{Liu_ent_vs_vol} 
for a summary with more details.

For $[f]\in\mathrm{Mod}(S)$ pseudo-Anosov, 
and for $\langle \mathrm{PD}(\Gamma),[S]\rangle=1$,
Cotton-Clay shows that $\mathrm{HP}(f,\Gamma)$
is nontrivial if and only if 
$\Gamma$ is represented by 
some $1$--periodic trajectory of the pseudo-Anosov suspension flow \cite[Theorem 3.6]{Cotton-Clay_sfh}.
Therefore,
if $S$ has genus $\geq3$, and if $\Gamma$ is not represented by 
any $1$--periodic trajectory of the pseudo-Anosov suspension flow.
we deduce $\widehat{\mathrm{HF}}(M_f,\mathfrak{s}_\theta+\Gamma)=0$, as asserted.
\end{proof}

With the above results, we finish the proof of Lemma \ref{no_fillable} as follows.

To argue by contradiction, suppose that $M_f$ supports a weakly fillable contact structure $\xi$
of real Euler class $\pm(e_f+2a)$ with $\langle a,[S]\rangle=1$.
Note that if $(W,\omega)$ is a weakly symplectic filling of $(M_f,\xi)$,
then $(W,-\omega)$ is a weakly symplectic filling of $(M_f,-\xi)$,
where $-\xi$ denotes the transverse-orientation reversal of $\xi$,
satisfying $e(-\xi)=-e(\xi)$.
Therefore, passing to $-\xi$ if necessary,
we may assume without loss of generality 
$$e(\xi)=-e_f-2a.$$
Denote by $\Gamma\in H_1(M_f;\Integral)$,
such that 
$$\mathfrak{s}_\xi=\bar{\mathfrak{s}}_\theta-\Gamma.$$
From $c_1(\bar{\mathfrak{s}}_\theta)=-c_1(\mathfrak{s}_\theta)=-e_f$ and $c_1(\mathfrak{s}_\xi)=e(\xi)=-e_f-2a$,
we obtain $\Gamma=\mathrm{PD}(a)$ in $H_1(M_f;\Real)$.

Applying Lemma \ref{fillable_implies_nonvanishing},
we infer 
$$\widehat{\mathrm{HF}}(-M_f,\mathfrak{s}_\xi)\neq0.$$

On the other hand, we can identify $-M_f$ with $+M_{f^{-1}}$,
so $\bar{\mathfrak{s}}_\theta$ is identified with
the distinguished $\mathrm{Spin}^{\mathtt{c}}$ structure of $M_{f^{-1}}$.
Our assumption implies that $-\Gamma$ in 
$H_1(M_f;\Integral)=H_1(M_{f^{-1}};\Integral)$
is not represented by any $1$--periodic trajectory of $f^{-1}$.
Applying Lemma \ref{no_trajectory_implies_vanishing} to $f^{-1}$,
and using $\mathfrak{s}_\xi=\bar{\mathfrak{s}}_\theta-\Gamma$,
we infer 
$$\widehat{\mathrm{HF}}(-M_f,\mathfrak{s}_\xi)=0,$$
contradicting the last inequality.

This completes the proof of Lemma \ref{no_fillable}.

\section{Periodic directions are virtually rare above boundary}\label{Sec-rare}

In this section, we build the technical heart of 
our construction for Theorem \ref{main_ecoc_fillable}.
This is the construction which allows us to virtually manipulate
the homology directions of periodic trajectories,
making them genercally rare 
on a wanted face of the Fried cone boundary (Lemma \ref{virtual_vacuum_directions}).

\begin{notation}\label{corner_A}
Suppose that $\mathcal{P}\subset V$ be a compact polytope in a finite-dimensional vector linear space.
Suppose that $F\subset\partial\mathcal{P}$ is a closed face and $v\in F$ is a vertex. 
Denote by
$$\mathcal{A}(v,F)\subset V$$
the set of all vectors $a\in V$,
such that $v+\epsilon\,a\in F$ holds for all sufficiently small $\epsilon\geq 0$.
This is a scaling-invariant closed subset of $V$,
which can be thought of as the \emph{corner} of $F$ at $v$ up to translation.
\end{notation}

\begin{lemma}\label{virtual_vacuum_directions}
Let $S$ be an oriented connected closed surface of genus $\geq2$,
and $f\colon S\to S$ be a pseudo-Anosov automorphism.
Suppose that $F\subset\partial\mathcal{B}_{\mathrm{Th}^*}(M_f)$
is a closed face with $e_f\in \partial F$ (Notation \ref{pA_notation}).
Then, there exist
some connected regular finite cover $\tilde{M}\to M_f$,
and some finite collection $\tilde{\mathcal{Z}}$ 
of real linear subspaces of $H^2(\tilde{M};\Real)$,
such that the following properties hold.
\begin{itemize}
\item
If $\tilde{\gamma}$ is any periodic trajectory of the pullback pseudo-Anosov flow on $\tilde{M}$,
such that the Poincar\'e dual $\mathrm{PD}([\tilde\gamma])$ 
lies in $\mathcal{A}(\tilde{e},\tilde{F})$,
then there is some $\tilde{Z}\in\tilde{\mathcal{Z}}$,
such that $\mathrm{PD}([\tilde\gamma])$ lies in $\tilde{Z}\cap\mathcal{A}(\tilde{e},\tilde{F})$.
\item
For any $\tilde{Z}\in\tilde{\mathcal{Z}}$,
the subset $\tilde{Z}\cap\mathcal{A}(\tilde{e},\tilde{F})$ in $\mathcal{A}(\tilde{e},\tilde{F})$
has positive codimension.
\end{itemize}
Here,
$\tilde{F}$ denotes the minimal closed face of $\mathcal{B}_{\mathrm{Th}^*}(\tilde{M})$ that contains the pullback of $F$
(see Lemma \ref{dual_ball_cover});
$\tilde{e}\in \partial\tilde{F}$ denotes the pullback of $e_f$,
which is again a vertex, dual to a distinguished fibered face;
see Notation \ref{corner_A} for $\mathcal{A}(\tilde{e},\tilde{F})\subset H^2(\tilde{M};\Real)$.
\end{lemma}

The rest of this section is devoted to the proof of Lemma \ref{virtual_vacuum_directions}.

First, we recall the following interpretation of 
Fried's characterization of the homology direction cone \cite[{Expos\'e 14}]{FLP_book}.

\begin{lemma}\label{Fried_cone_characterization}
For any pseudo-Anosov automorphism $f\colon S\to S$ on an oriented closed surface of genus $\geq2$,
$$\mathrm{PD}\left(\mathcal{A}\left(e_f;\mathcal{B}_{\mathrm{Th}^*}(M_f)\right)\right)=\mathcal{C}_{\mathrm{Fr}}(f).$$
Here, $\mathcal{C}_{\mathrm{Fr}}(f)\subset H_1(M_f;\Real)$ denotes the Fried cone
of homology directions, namely,
the closure of all the rays emanating from the origin
and passing through $[\gamma]$,
where $\gamma$ ranges over all the periodic trajectories 
of the pseudo-Anosov suspension flow on $M_f$.
See also Notations \ref{pA_notation} and \ref{corner_A}.
\end{lemma}

\begin{proof}
By Fried's characterization,
$\Gamma\in H_1(M_f;\Real)$ lies $\mathcal{C}_{\mathrm{Fr}}(f)$ if and only if
the inequality
$\langle \mathrm{PD}(\Gamma),\Sigma\rangle\geq0$ holds for all $\Sigma\in H_2(M_f;\Real)$
which lies in the distinguished fibered cone $\mathcal{C}_{\mathrm{Th}}(f)$
(that is, the open cone over the distinguished fibered face of 
the Thurston norm unit ball $\mathcal{B}_{\mathrm{Th}}(M_f)$ dual to $e_f$);
see \cite[{Expos\'e 14}, Theorem 14.11]{FLP_book}.
In $H^2(M_f;\Real)$, the intersection of all the half-spaces
defined by the inequalities $\langle x-e_f,\Sigma\rangle \geq 0$,
ranging over all $\Sigma\in\mathcal{C}_{\mathrm{Th}}(f)$,
is exactly $e_f+\mathcal{A}\left(e_f;\mathcal{B}_{\mathrm{Th}^*}(M_f)\right)$,
implying the asserted characterization.
\end{proof}

\begin{lemma}\label{Fried_cone_more}
	Let $f\colon S\to S$ be a pseudo-Anosov automorphism
	on an orientable connected closed surface of genus $\geq2$.
	Suppose that $F\subset \mathcal{B}_{\mathrm{Th}^*}(M_f)$
	is a closed face with $e_f\in F$.
	\begin{enumerate}
	\item Suppose $e_f\in\partial F$.
	Then, $\mathrm{PD}(\mathcal{A}(e_f;F))$ contains the real homology class $[\gamma]$ 
	of some periodic trajectory $\gamma$.
	\item Suppose $F\subset \partial\mathcal{B}_{\mathrm{Th}^*}(M_f)$.
	If $H$ is any finitely generated subgroup of $\pi_1(M_f)$,
	and if, for all $g\in H$,	
	the real homology class $[g]$ lies in the real linear subspace of $H_1(M_f;\Real)$
	spanned by $\mathrm{PD}(\mathcal{A}(e_f;F))$,
	then $H$ is a quasiconvex subgroup of infinite index in $\pi_1(M_f)$.
	\end{enumerate}
See Notations \ref{pA_notation} and \ref{corner_A}.
\end{lemma}

\begin{proof}
	The assumption of the first statement says 
	that $F$ has positive dimension.	
	All higher dimensional cases can be 
	reduced to the $1$--dimensional case, as $F$ is closed.
	Every $1$--dimensional face of $\mathcal{C}_{\mathrm{Fr}}(M_f)$
	must contain the real homology class of some periodic trajectory.
	In fact, this is part of a more detailed picture of Fried's characterization;
	for example, this follows immediately of \cite[Lemma 5.13]{Liu_vhsr}.
	Then we can translate the conclusion 
	into the first statement by Lemma \ref{Fried_cone_characterization}.
	
	The assumption of the second statement says 
	that $F$ has positive codimension in $\mathcal{B}_{\mathrm{Th}^*}(M_f)$.
	Since $\mathcal{B}_{\mathrm{Th}^*}(M_f)$ is the convex hull of integral lattice points,
	there exists some nonzero, primitive $\Sigma\in H_2(M_f;\Integral)$, 
	such that $\langle a,\Sigma\rangle=0$	holds for all $a\in \mathcal{A}(e_f;F)$.
	In particular, $\Sigma$ lies on the boundary of the distinguished fibered cone $\mathcal{C}_{\mathrm{Th}}(f)$.
	Therefore, $\mathrm{PD}(\Sigma)$ cannot a fibered class.
	
	Denote by $\psi\colon \pi_1(M_f)\to \Integral$ 
	the group homomorphism $g\mapsto \langle \mathrm{PD}([g]),\Sigma\rangle$.
	By assumption, $H$ lies in the kernel of $\psi$.
	Since $\psi$ has infinite cyclic image, $H$ must be of infinite index in $\pi_1(M_f)$.
	Moreover, $H$ cannot contain a normal surface subgroup $N$ of $\pi_1(M_f)$.
	For otherwise, $H/N$ would be a subgroup of 
	the virtually infinite cyclic group $\pi_1(M_f)/N$,
	implying that the kernel of $\psi$ is finitely generated.
	This cannot be the case when $\mathrm{PD}(\Sigma)$ is not a fibered class,
	(see \cite{Stallings-fibering}, or \cite[{Expos\'e 14, Theorem 14.2}]{FLP_book}).
	It follows that $H$ can only be a quasi-convex in the word-hyperbolic group $\pi_1(M_f)$,
	due to a fundamental dichotomy in the theory of closed hyperbolic $3$--manifold groups,
	(see \cite[Theorem 4.1.2 and Proposition 4.4.2]{AFW_book_group}).
\end{proof}

\begin{lemma}\label{facial_clusters}
	Let $f\colon S\to S$ be a pseudo-Anosov automorphism
	on an orientable connected closed surface of genus $\geq2$.
	Suppose that $F\subset \mathcal{B}_{\mathrm{Th}^*}(M_f)$
	is a closed face with $e_f\in F$.
	Then, there exists some finite collection $\mathcal{H}$	of 
	finitely generated subgroups of $\pi_1(M_f)$
	with the following properties.
	\begin{itemize}	
	\item
	If $\gamma$ is a periodic trajectory of the suspension flow,	
	such that $[\gamma]$ lies in $\mathrm{PD}(\mathcal{A}(e_f;F))$,
	then some finite cyclic cover of $\gamma$, as a free-homotopy loop,
	represents a conjugacy class in $\pi_1(M_f)$
	which has nonempty intersection with some $H\in\mathcal{H}$.
	\item
	For all $H\in\mathcal{H}$, and for all $g\in H$,
	the real homology class $[g]\in H_1(M_f;\Real)$ lies
	in the real linear subspace spanned by $\mathrm{PD}(\mathcal{A}(e_f;F))$.
	\end{itemize}
See Notations \ref{pA_notation} and \ref{corner_A}.
\end{lemma}


\begin{proof}
	We extract the asserted subgroups using 
	symbolic dynamics associated to a Markov partition of the pseudo-Anosov automorphism.
	This relies on the techniques developed in \cite[Section 5 and 6]{Liu_vhsr}.
	A more detailed exposition of a very similar construction
	appears in \cite[Lemma 6.6]{Liu_profinite_almost_rigidity}.
	
	In a nutshell, 
	upon choice of any Markov partition $\mathcal{R}$ with respect to $(S,f)$,
	we are able to obtain a finite directed graph $T_{f,\mathcal{R}}$
	(as a finite cell $1$--complex with oriented $1$--cells), called the \emph{transition graph}.
	This comes together with a continuous map $T_{f,\mathcal{R}}\to M_f$,
	to be described as follows.
	
	An $m$--\emph{dynamical cycle} of a directed graph refers to 
	a $m$--tuple of directed edges $(e_i)$ indexed by $i\in\Integral/m\Integral$,
	such that the terminal vertex of each $e_i$ agrees the initial vertex of $e_{i+1}$.
	
	The first property of $T_{f,\mathcal{R}}\to M_f$ 
	is that every $m$--dynamical cycle of $T_{f,\mathcal{R}}$
	projects an $m$--periodic trajectory in $M_f$, up to free homotopy,
	and moreover, every $m$--periodic trajectory in $M_f$ 
	admits a finite cyclic cover which arises this way.
	
	The second property of $T_{f,\mathcal{R}}\to M_f$ 
	is basically a refinement of the first property to each radial face $E$
	of the Fried cone $\mathcal{C}_{\mathrm{Fr}}(f)$ (see Lemma \ref{Fried_cone_characterization}).
	To each $E$,	
	one can uniquely associate a possibly disconnected, directed subgraph $T_{f,\mathcal{R}}[E]$ of $T_{f,\mathcal{R}}$,
	called the \emph{support subgraph} for $E$.
	The image of $H_1(T_{f,\mathcal{R}}[E];\Real)\to H_1(M_f;\Real)$
	is equal to the real linear subspace of $H_1(M_f;\Real)$ spanned by $E$,
	and moreover, every periodic trajectory in $M_f$ of homology class in $E$
	arises from a dynamical cycle in $T_{f,\mathcal{R}}[E]$, 
	up to passage to a finite cyclic cover.
	
	The above description of the transition graph and the support subgraphs
	follows from \cite[Lemmas 5.7, 5.10, 6.4, and Remark 5.12]{Liu_vhsr};
	the reader is referred to the proof of \cite[Lemma 6.6]{Liu_profinite_almost_rigidity} 
	for an elaboration.
	
	To finish the proof,
	we take $E$ to be the radial face $\mathrm{PD}(\mathcal{A}(e_f;F))$ of $\mathcal{C}_{\mathrm{Fr}}(f)$
	(Lemma \ref{Fried_cone_characterization}).
	For each connected component $V$ of $T_{f,\mathcal{R}}[E]$,
	we obtain a finitely generated subgroup $H_V$ of $\pi_1(M_f)$ as the image of $\pi_1(V)\to\pi_1(M_f)$,
	(fixing a conjugate by choosing auxiliary base points).
	The asserted finite collection $\mathcal{H}$ consists of all $H_V$,
	indexed by all $V\in\pi_0(T_{f,\mathcal{R}}[E])$.
\end{proof}

For any covering projection $Y'\to Y$ and any continuous map $X\to Y$ 
of a path-connected space, locally path-connected space $X$,
the restriction $X'\to Y'$ of the pullback map $X\times_Y Y'\to Y'$ 
to any path-connected component $X'$ of $X\times_Y Y'$
is called an \emph{elevate} of $X\to Y$ to $Y'$.

\begin{lemma}\label{virtual_homological_separation}
	Let $M$ be a closed hyperbolic $3$--manifold, and $\gamma$ 
	be a homotopically nontrivial loop $S^1\to M$.
	Suppose that $H$ is a infinite-index, quasiconvex subgroup of $\pi_1(M)$.
	Then,	there exist some connected regular finite cover $\tilde{M}$ of $M$
	(fixing a base point and a lift), 
	and 
	some deck-invariant finite collection $\tilde{\mathcal{Z}}$ of real linear subspaces of $H^2(\tilde{M};\Real)$, 
	such that the following properties hold.
	\begin{itemize}
	\item 
	For every conjugate $K$ of $H$ in $\pi_1(M)$, and for every $g\in K\cap\pi_1(\tilde{M})$,
	there is some $\tilde{Z}\in\tilde{\mathcal{Z}}$,
	such that	$[g]\in H_1(\tilde{M};\Real)$ lies in $\mathrm{PD}(\tilde{Z})$.	
	\item 
	For every $\tilde{Z}\in \tilde{\mathcal{Z}}$, 
	there is some elevate $\tilde{\gamma}$ of $\gamma$,
	such that $[\tilde{\gamma}]\in H_1(\tilde{M};\Real)$
	does not lie in $\mathrm{PD}(\tilde{Z})$.
	\end{itemize}
\end{lemma}

\begin{proof}
	We rewrite $\Pi=\pi_1(M)$ for convenience, 
	and treat $\gamma$ as a nontrivial conjugacy class of elements in $\Pi$.
	Since $H$ is quasiconvex of infinite index in $\Pi$, 
	we can find some representative $x\in\Pi$ of $\gamma$,
	such that $H$ and $x$ generate a free-amalgam, quasiconvex subgroup $H*\langle x\rangle$ of $\Pi$.
	For example, this can be arranged by picking a representative $x$ with axis 
	disjoint from the convex hull of $H$ in $\mathbb{H}^3$, 
	and showing the free amalgamation property 
	is an easy exercise of the ping pong argument.
	
	Because $\Pi$ is word-hyperbolic and virtually compact special \cite{Agol_VHC},
	any quasiconvex subgroup of $\Pi$ is a virtual retract
	\cite[Theorem 7.3]{Haglund--Wise}, (see also \cite[Proof 2 of Theorem 4.13]{Wise_book}).
	In particular, this means that we can find some finite-index subgroup $\Pi'$ of $\Pi$
	containing $H*\langle x\rangle$, and some group homomorphism $\Pi'\to H*\langle x\rangle$
	which is the identity restricted to $H*\langle x\rangle$. 
	
	Take $M'$ to be the connected finite cover of $M$ (with a lifted base point) 
	corresponding to some $\Pi'$ as above. 
	Obtain $W'$ as the real linear subspace of $H^2(M';\Real)$,
	such that $\mathrm{PD}(W')=\mathrm{Span}_\Real\{[g]\colon g\in H\}$ in $H_1(M';\Real)$.
	By construction,
	the conjugacy class of $x$ in $\Pi'$ represents an elevate $\gamma'$ of $\gamma$ to $M'$
	(which is a lift in this case). Since $\Pi'$ retracts onto $H*\langle x\rangle$,
	$H_1(M';\Real)$ contains $\mathrm{PD}(W')\oplus \Real\cdot[\gamma']$ as a direct summand,
	implying $[x]\not\in\mathrm{PD}(W')$.
	Moreover, $[g]\in \mathrm{PD}(W')$ holds for every $g\in H\cap\Pi'=H$.
	
	Take $\tilde{M}$ to be 
	the connected regular finite cover of $M$ corresponding to 
	the intersection $\tilde{\Pi}$ of all the conjugates of $\Pi'$ in $\Pi$. 
	Obtain $\tilde{W}$ as the preimage of $W'$
	under the umkehr homomorphism $H^2(\tilde{M};\Real)\to H^2(M';\Real)$.
	In other words, 
	$\mathrm{PD}(\tilde{W})$ is the preimage of $\mathrm{PD}(\tilde{W})$ under 
	the induced homomorphism $H_1(\tilde{M};\Real)\to H_1(M';\Real)$.
	
	Take $\tilde{\mathcal{Z}}$ 
	to be the orbit of $\tilde{W}$ under the action of
	the deck transformation group.
	
	We verify the asserted property for the pair $(\tilde{M},\tilde{\mathcal{Z}})$
	as constructed above.
	Note that $\tilde{M}$ is regular over $M$, and $\tilde{\mathcal{Z}}$
	is invariant under deck transformations.
	By symmetry, it suffices to check the asserted properties
	for $K=H$ and for $\tilde{Z}=\tilde{W}$.
	By construction, any $g\in H\cap \tilde{\Pi}$ represents a real homology class 
	$[g]\in\mathrm{PD}(\tilde{W})$.
	Moreover,	
	$\mathrm{PD}(\tilde{W})$ does not contain $[\tilde{\gamma}]$,
	where we can take $\tilde{\gamma}$ to be any elevate of $\gamma'$ to $\tilde{M}$.
	This is because $\mathrm{PD}(W')$ does not contain $[\gamma']$.
	Therefore, 
	the asserted properties hold for $K=H$ and for $\tilde{Z}=\tilde{W}$,	as desired.
\end{proof}

With the above preparation, 
we finish the proof of Lemma \ref{virtual_vacuum_directions} as follows.

Let $f\colon S\to S$ be a pseudo-Anosov automorphism
on an orientable connected closed surface of genus $\geq2$.
Suppose that $F\subset \mathcal{B}_{\mathrm{Th}^*}(M_f)$
is a closed face with $e_f\in F$.

Associated to $F$,
there exists a finite collection 
$\mathcal{H}$ of finitely generated subgroup of $\pi_1(M_f)$, 
as asserted in Lemma \ref{facial_clusters}.
By Lemma \ref{Fried_cone_more},
each $H\in\mathcal{H}$ is quasiconvex of infinite index in $\pi_1(M_f)$.
Also by Lemma \ref{Fried_cone_more}, 
$\mathrm{PD}(\mathcal{A}(e_f;F))$ must contain
the real homology class $[\eta]$ of some periodic trajectory $\gamma$.
Fix some $\mathcal{H}=\{H_1,\cdots,H_k\}$ and some $\eta$ as above.

For each $i=1,\cdots,k$, 
we obtain a connected regular finite cover $M'_i\to M_f$,
together with a finite collection $\mathcal{Z}'_i$ of real linear subspaces of $H^2(M'_i;\Real)$,
by applying Lemma \ref{virtual_homological_separation} 
(setting $(M,H,\gamma)=(M_f,H_i,\eta)$).

Take $\tilde{M}$ to be the connected regular finite cover of $M_f$
corresponding to the intersection of all $\pi_1(M'_i)$ in $\pi_1(M_f)$.
For each $i=1,\cdots,k$, and for each $Z'\in\mathcal{Z}'_i$, 
we obtain a real linear subspace $\tilde{Z}$ of $H^2(\tilde{M};\Real)$,
as the the preimage of $Z'\in\mathcal{Z}'_i$ under 
the umkehr homomorphism $H^2(\tilde{M};\Real)\to H^2(M'_i;\Real)$.

Take $\tilde{\mathcal{Z}}$ to be the finite collection of real linear subspaces of $H^2(\tilde{M};\Real)$
consisting of all $\tilde{Z}$.

We verify that the pair $(\tilde{M},\tilde{\mathcal{Z}})$ as constructed above
satisfies the asserted property in Lemma \ref{virtual_vacuum_directions},
as follows.

Every periodic trajectory $\tilde{\gamma}$ of the pullback pseudo-Anosov flow on $\tilde{M}$
projects a periodic trajectory $\gamma$ in $M_f$.
Suppose $\mathrm{PD}([\tilde{\gamma}])$ lies in $\mathcal{A}(\tilde{e},\tilde{F})$.
Then, $\mathrm{PD}([\gamma])$ lies in $\mathcal{A}(e_f,F)$,
by Lemma \ref{dual_ball_cover} and Notation \ref{corner_A}.
There is some $H_i$,
such that some finite cyclic cover of $\gamma$ 
represents a conjugacy class of $\tilde{M}_f$ intersecting $H_i$, 
as guaranteed by Lemma \ref{facial_clusters}.
It follows that $\tilde{\gamma}$ is an elevate of some periodic trajectory $\gamma'$ in $M'_i$,
such that $\gamma'$ represents a conjugacy class of $\pi_1(M'_i)$ which intersects $H_i\cap\pi_1(M'_i)$.
Therefore, $[\gamma']$ is contained in $\mathrm{PD}(Z')$ for some $Z'\in\mathcal{Z}'_i$,
as guaranteed by Lemma \ref{virtual_homological_separation}.
It follows that $[\tilde{\gamma}]$ maps into $\mathrm{PD}(Z')$ 
under the induced homomorphism $H_1(\tilde{M};\Real)\to H_1(M'_i;\Real)$.
By construction, $\mathrm{PD}[\tilde{\gamma}]$ lies in $\tilde{Z}\cap\mathcal{A}(\tilde{e},\tilde{F})$,
where $\tilde{Z}\in\tilde{\mathcal{Z}}$ is the umkehr-homomorphism preimage of $Z'$.
Therefore,
$(\tilde{M},\tilde{\mathcal{Z}})$
satisfies the first asserted property in Lemma \ref{virtual_vacuum_directions}.

Every $\tilde{Z}\in\tilde{\mathcal{Z}}$ maps onto some $Z'\in\mathcal{Z}'_i$ for some $\mathcal{Z}'_i$
under the umkehr homomorphism $H^2(\tilde{M};\Real)\to H^2(M'_i;\Real)$.
Meanwhile, $\mathcal{A}(\tilde{e},\tilde{F})$ maps onto
$\mathcal{A}(e'_i,F'_i)$, by Lemma \ref{dual_ball_cover} and Notation \ref{corner_A}.
The subset $Z'\cap \mathcal{A}(e'_i,F'_i)$ in $\mathcal{A}(e'_i,F'_i)$ has positive codimension,
since it does not contain $\mathrm{PD}([\eta'])\in\mathcal{A}(e'_i,F'_i)$
for some elevate $\eta'$ of $\eta$ to $M'_i$,
as guaranteed by Lemma \ref{virtual_homological_separation}.
It follows that the subset $\tilde{Z}\cap\mathcal{A}(\tilde{e},\tilde{F})$ in 
$\mathcal{A}(\tilde{e},\tilde{F})$ also has positive codimension.
Therefore,
$(\tilde{M},\tilde{\mathcal{Z}})$
satisfies the second asserted property in Lemma \ref{virtual_vacuum_directions}.

This completes the proof of Lemma \ref{virtual_vacuum_directions}.

\section{Proof of the main theorem}\label{Sec-main_proof}
In this section, we prove Theorem \ref{main_ecoc_fillable}.

Before the proof, we need one more lemma as follows,
which has the effect of promoting
any rational point of the real second cohomology
to an actual even lattice point.
The lemma says that this can be done 
by pulling back to finite cyclic covers 
dual to a primitive integral second homology class,
if you use, for example, a fibered class.

\begin{lemma}\label{iterate}
Let $M$ be an orientable connected closed $3$--manifold.
Suppose that $\Sigma\in H_2(M;\Integral)$ is a primitive integral homology class,
such that the first Betti numbers $b_1(M'_m)$
of the $m$--cyclic covers of $M$ dual to $\Sigma$ are uniformly bounded
for all $m\in\Natural$.

Then, there exists some $l\in\Natural$, depending on $M$ and $\Sigma$,
such that the following property hold.
\begin{itemize}
\item
For any $r\in\Natural$ and $a\in H^2(M;\Rational)$,
if $\langle a,\Sigma\rangle$ is an even integer, 
and if $ra$ is an even lattice point,
then for any $m\in\Natural$ divisible by $lr$,
the pullback $a'_m$ of $a$ to $M'_m$ is an even lattice point
in $H^2(M'_m;\Rational)$.
\end{itemize}
\end{lemma}

\begin{proof}
By assumption, there exists some $l\in\Natural$,
such that $b_1(M'_l)$ achieves the maximum among all $b_1(M'_m)$.
We show that any such $l$ satisfies the asserted property.

To this end, suppose $r\in\Natural$, $s\in\Integral$ and $a\in H^2(M;\Rational)$,
such that $\langle a,\Sigma\rangle = 2s$, and $ra$ is an even lattice point.
Fix an orientation of $M$,
and fix a reference $\mathrm{Spin}^{\mathtt{c}}$ structure $\mathfrak{s}_0$ of $M$.
Note that any oriented closed $3$--manifold admits a torsion $\mathrm{Spin}^{\mathtt{c}}$ structure.
This fact is trivial when the first Betti number is zero;
when the first Betti number is nonzero, every vertex of the dual Thurston norm unit ball
is an even lattice point, and can be realized as the real first Chern class
of some $\mathrm{Spin}^{\mathtt{c}}$ structure 
which comes from a transversely oriented taut foliation;
then, a torsion $\mathrm{Spin}^{\mathtt{c}}$ structure 
can be obtained from that one by adding a suitable first homology class.
We assume $\mathfrak{s}_0$ to be torsion.

As $\Sigma\in H_2(M;\Integral)$ is primitive,
we may represent $\Sigma$ with some orientable connected closed subsurface of $M$.
This surface lifts to each $m$--cyclic cover $M'_m$ dual to $\Sigma$.
We denote by $\Sigma'_m\in H_2(M'_m;\Integral)$ the homology class of any lifted surface,
which does not depend on the choice of lifts.
Since $\mathfrak{s}_0$ is torsion, we observe
$$c_1(\mathfrak{s}'_m)=0$$
in $H^2(M'_m;\Rational)$ for all $m\in\Natural$,
where $\mathfrak{s}'_m$ denotes the pullback of $\mathfrak{s}_0$ to $M'_m$.

Since $ra$ is an even lattice point, the pullback class $ra'_l\in H^2(M'_l;\Rational)$
is also an even lattice point, so there exists some $\Gamma\in H_1(M'_l;\Integral)$,
such that 
$$ra'_l=c_1\left(\mathfrak{s}'_l+\Gamma\right)=2\cdot\mathrm{PD}(\Gamma)$$
holds in $H^2(M'_l;\Rational)$.
We obtain 
$$\langle \mathrm{PD}(\Gamma),\Sigma'_l\rangle=\langle (r/2)\cdot a'_l,\Sigma'_l\rangle=rs.$$
Representing $\Gamma$ by any loop $\gamma$ in $M'_l$,
we see that $\gamma$ has totally $r$ distinct lifts to $M'_{lr}$,
since $s$ is an integer.
Because $b_1(M'_{lr})\geq b_1(M'_l)$, 
the maximality of $b_1(M'_l)$ implies $b_1(M'_{lr})=b_1(M'_l)$.
In this case,
the $r$--cyclic deck tranformation group acts trivially on $H_1(M'_{lr};\Rational)$.
It follows that the $r$ different lifts of $\gamma$ are all homologous modulo torsion,
giving rise to one and the same integral lattice point $\Gamma'\in H_1(M'_{lr};\Rational)$.
So,
the pullback of $\mathrm{PD}(\Gamma)$ in $H^2(M'_{lr};\Rational)$ is equal to $r\cdot\mathrm{PD}(\Gamma')$.
On the other hand, 
the pullback of $\mathrm{PD}(\Gamma)=(r/2)\cdot a'_l$ is $(r/2)\cdot a'_{lr}$ in $H^2(M'_{lr})$.
So, we obtain $(r/2)\cdot a'_{lr}=r\cdot\mathrm{PD}(\Gamma')$,
or equivalently,
$$a'_{lr}=2\cdot\mathrm{PD}(\Gamma').$$
This shows that $a'_{lr}$ is an even lattice point in $H^2(M'_{lr};\Rational)$.
Consequently, for all $m\in\Natural$ divisible by $lr$, 
the further pullbacks $a'_m$ in $H^2(M'_m;\Rational)$ are all even lattice points.
\end{proof}

We are ready to prove Theorem \ref{main_ecoc_fillable}, as follows.

Let $M$ be any oriented closed hyperbolic $3$--manifold.
Our goal is to construct some connected finite cover $\tilde{M}$ of $M$,
together with some even lattice point $w\in H^2(\tilde{M};\Real)$
on the boundary $\partial\mathcal{B}_{\mathrm{Th}^*}(\tilde{M})$ of the dual Thurston norm unit ball,
such that $\tilde{w}$ is not the real Euler class of 
any weakly symplectically fillable contact structure on $\tilde{M}$.

We construct the asserted $(\tilde{M},\tilde{w})$ 
through a tower of three connected finite covers over $M$:
$$\tilde{M}\to M''\to M'\to M.$$ 

First, we obtain a connected finite cover $M'$ of $M$ which fibers over a circle.
Moreover, we require $b_1(M')\geq 2$. 
The existence of $M'$ follows from 
confirmed virtual properties of finite-volume hyperbolic $3$--manifolds \cite{Agol_VHC}.
Since $b_1(M')\geq2$,
any Thurston norm fibered cone contains infinitely many primitive classes, 
and their Thurston norm can be arbitrarily large.
Therefore,
we can find some oriented connected closed fiber surface $S'$ in $M'$ of genus $\geq3$.
The manifold $M'$ can be identified with the mapping torus of 
some pseudo-Anosov automorphism $S'\to S'$,
which is isotopic to the monodromy of the fibering.
Below, we speak of periodic trajectories in $M'$
with respect to the suspension pseudo-Anosov flow 
under the identification.

Fix some $S'\subset M'$ as above. Denote by $e'\in H^2(M';\Real)$
the real Euler class of the oriented plane distribution tangent to the surface fibers.
By construction, 
$e'$ must be a vertex on $\partial\mathcal{B}_{\mathrm{Th}^*}(M')$.
In particular, $e'$ is an even lattice point of $H^2(M';\Real)$.
Moreover, we note
$$\langle e',[S']\rangle = \chi(S')\leq -4.$$

Fix some closed face $F'\subset\partial\mathcal{B}_{\mathrm{Th}^*}(M')$
with $e'\in \partial F'$, which exists since $b_1(M')\geq2$.

Next, we obtain a connected finite cover $M''$ of $M'$ as asserted in Lemma \ref{virtual_homological_separation},
with respect to $e'$ and $F'$. Denote by $e''\in H^2(M'';\Real)$ the pullback of $e'$,
which is a vertex of $\mathcal{B}_{\mathrm{Th}^*}(M'')$.
Denote by $F''$ the minimal closed face of $\mathcal{B}_{\mathrm{Th}^*}(M'')$
that contains the pullback of $F'$.
Hence, $e''\in \partial F''$ and $F''\subset\partial \mathcal{B}_{\mathrm{Th}^*}(M'')$.
Denote by $S''\subset M''$ any preimage component of $S'$.

We claim that there is some rational point $a''\in H^2(M'';\Real)$
with $$\langle a'',[S'']\rangle =1,$$
such that $e''+2a''$ lies in $F''$, and 
such that no positive integral multiple of $\mathrm{PD}(a'')\in H_1(M'';\Real)$ is represented
by any periodic trajectory of the pullback pseudo-Anosov flow on $M''$.

In fact, as guaranteed by Lemma \ref{virtual_homological_separation},
there is a finite collection $\mathcal{Z}''$ of real linear subspaces of $H^2(M'';\Real)$,
such that the complement $U$ of their union in $\mathcal{A}(e'';F'')$ (Notation \ref{corner_A})
contains no $\mathrm{PD}([\gamma''])$ for any periodic trajectory $\gamma''$ in $M''$.
In particular, $U$ is a scaling-invariant, open, dense subset of $\mathcal{A}(e'';F'')$.
We observe
$$F''=\left(e''+\mathcal{A}(e'';F'')\right)\cap\mathcal{B}_{\mathrm{Th}^*}(M''),$$
so $(e''+U)\cap F''$ is open and dense in $F''$.
We also note that the conclusion of Lemma \ref{virtual_homological_separation}
forces $F''$ to have dimension $\geq2$ (see Lemma \ref{Fried_cone_more}).
On the other hand, the vertices of $F''$ are all even lattice points in $H^2(M'';\Real)$.
Since $e''$ is a vertex of $F''$ with $\langle e'',[S'']\rangle=\chi(S'')$,
and since any point $q''\in F''$ other than $e''$ satisfies $\langle q'',[S'']\rangle >\chi(S'')$,
the points $q''\in F''$ with $\langle q'',[S'']\rangle = \chi(S)+2$
form a compact polytope $K$ in $H^2(M'';\Real)$ with rational vertices.
Since $F''$ has dimension $\geq2$, $K$ must have dimension $\geq1$.
Since is open and dense in $K$, the rational points must be dense in $(e''+U)\cap K$.
Every rational point $q''$ in $(e''+U)\cap K$ gives rise to a rational point $a''=(q''-e'')/2$ as claimed.

Fix some $a''\in H^2(M'';\Real)$ as claimed above.

Finally, we obtain a connected $m$--cyclic cover $\tilde{M}$ of $M''$ dual to $[S'']$,
such that the pullback $\tilde{a}\in H^2(\tilde{M};\Real)$ of $a''$ is an even lattice point.
The condition can be satisfied as long as we pick some sufficiently divisible $m\in\Natural$
(depending on $M''$, $[S'']$ and $a''$), by Lemma \ref{iterate}.
Fix a lift $\tilde{S}$ of $S''$ to $\tilde{M}$.
Denote by $\tilde{e}\in H^2(\tilde{M};\Real)$ the pullback of $e''$.
We obtain 
$\langle \tilde{a},m[\tilde{S}]\rangle=m\cdot\langle a'',[S'']\rangle=m$,
implying
$$\langle \tilde{a},[\tilde{S}]\rangle =1.$$
Denote by $\tilde{F}$
the minimal closed face of $\mathcal{B}_{\mathrm{Th}^*}(\tilde{M})$ that contains the pullback of $F''$.
Hence, $\tilde{e}\in\partial\tilde{F}$, 
and $\tilde{F}\subset\partial \mathcal{B}_{\mathrm{Th}^*}(\tilde{M})$.

Set
$$\tilde{w}=\tilde{e}+2\tilde{a}.$$ 

To summarize, we have constructed a connected finite cover $\tilde{M}$ of $M$,
together with an even lattice point $\tilde{w}\in H^2(\tilde{M};\Real)$ 
which lies in a closed face $\tilde{F}\subset \partial\mathcal{B}_{\mathrm{Th}^*}(\tilde{M})$.

It remains to check that $\tilde{w}$ is not the real Euler class
of any weakly symplectically fillable contact structure on $\tilde{M}$.
To this end, we observe that $\mathrm{PD}(\tilde{a})$ 
is not represented by any $1$--periodic trajectory of the pullback pseudo-Anosov flow on $\tilde{M}$.
This is because $\mathrm{PD}(\tilde{a})$ projects $m\cdot\mathrm{PD}(a'')$
under the induced homomorphism $H_1(\tilde{M};\Real)\to H_1(M'';\Real)$,
but $m\cdot\mathrm{PD}(a'')$ is not represented by any periodic trajectory in $M''$.
Since $\langle \tilde{a},[\tilde{S}]\rangle=1$,
and since $\tilde{S}$ has genus $\geq3$ as $S'$ does,
we can apply Lemma \ref{no_fillable},
concluding the nonexistence of any weakly symplectically fillable contact structures 
of real Euler class $\tilde{w}=\tilde{e}+2\tilde{a}$.

This completes the proof of Theorem \ref{main_ecoc_fillable}.

\section{Further discussion}\label{Sec-discussion}

For any oriented closed $3$--manifold $M$, 
and 
for every even lattice point $w\in H^2(M;\Real)$ of dual Thurston norm $1$
(or more generally, $\leq1$),
Yazdi asks whether there exists some finite cover $M'$ of $M$,
such that the pullback $a'$ of $a$ to $M'$ is the real Euler class
of some transversely oriented taut foliation on $M'$ \cite[Question 9.4]{Yazdi_ecoc}.
An expected affirmative answer to this question 
is called the \emph{virtual Euler class one conjecture}.

We explain 
why our construction for Theorem \ref{main_ecoc_fillable}
does not negate Yazdi's conjecture.
Retaining the notations $(\tilde{M},\tilde{w})$ as in Section \ref{Sec-main_proof},
one might attempt to argue that
for any finite cover $M'''$ of $\tilde{M}$,
the pullback $w'''$ of $\tilde{w}$ to $M'''$
is not realized by taut foliations (or fillable contact structures) any more.
The argument works
when $M'''$ is cyclic over $\tilde{M}$ 
dual to the constructed surface fiber $\tilde{S}$.
However, if $\tilde{S}$ does not lift to $M'''$,
a preimage component $S'''$ of $\tilde{S}$ in $M'''$ 
will cover $\tilde{S}$ of some degree $k>1$.
One obtains $w'''=e'''+2a'''$ and $\langle a''',[S''']\rangle =k>1$.
Our construction implies that $\mathrm{PD}(a''')$ 
cannot be represented by any $k$--periodic tranjectory in $M'''$,
but $\mathrm{PD}(a''')$ may still be the sum of 
several homology classes that are representable 
by periodic trajectories (with sum of periods $k$).
In that case, 
certain similar tuples of periodic trajectories 
may present as chain complex generators of the periodic Floer homology.
In particular,
a similar criterion as Lemma \ref{no_trajectory_implies_vanishing}
would fail for $k$--periodic trajectories with $k>1$.

It might be true that 
our counter-examples in Theorem \ref{main_ecoc_fillable}
are not realized by any tight contact structure either.
See \cite[Remark 1.3]{Sivek--Yazdi} for discussion regarding
the Euler class one conjecture for tight contact structures.

\bibliographystyle{amsalpha}

\end{document}